\documentclass[11pt]{article}
\usepackage{amsfonts}
\usepackage{titlesec}
\usepackage{fullpage}
\usepackage{amsfonts}
\usepackage{amsmath}
\usepackage{amssymb}
\usepackage{amsbsy}
\usepackage{amsthm}
\usepackage{epsfig}
\usepackage{graphicx}
\usepackage{colordvi}
\usepackage{graphics}
\usepackage{color}
\usepackage{bm}

\usepackage{tikz}
\usepackage{color,cancel}
\usepackage[colorlinks=true]{hyperref}
\usepackage{url}
\usepackage{verbatim}
\hypersetup{urlcolor=blue, citecolor=red}

\numberwithin{equation}{section}

\titleformat{\section}[runin]
  {\normalfont\Large\bfseries}{\thesection}{1em}{}
\titleformat{\subsection}[runin]
  {\normalfont\large\bfseries}{\thesubsection}{1em}{}
\newtheorem{theorem}{Theorem}[section]

\newtheorem{lemma}{Lemma}[section]

\newtheorem{algorithm}{Algorithm}[section]

\def\grad{{\nabla}}

\usepackage[maxfloats=80]{morefloats}
\begin{document}
\title{\textbf{\large\MakeUppercase{ Second order ensemble simulation for MHD flow in Els\"asser variable with noisy input data}}}

\vspace{-20mm}
\author{\MakeUppercase{
Muhammad Mohebujjaman}
\thanks{Department of Mathematics, Virginia Tech, Blacksburg, VA, 24061, USA;
\url{jaman@vt.edu}}
}
\date{}
\maketitle
\vspace{-7mm}
\small\textbf{Abstract:} We propose, analyze and test a fully discrete, efficient second-order algorithm for computing flow ensembles average of viscous, incompressible, and time-dependent magnetohydrodynamic (MHD) flows under uncertainties in initial conditions. The scheme is decoupled and based on Els\"asser variable formulation. The algorithm uses the breakthrough idea of Jiang and Layton, 2014 to approximate the ensemble average of $J$ realizations. That is, at each time step, each of the $J$ realization shares the same coefficient matrix for different right-hand side matrices. Thus, storage requirements and computational time are reduced by building preconditioners once per time step and reuse them. We prove stability and optimal convergence with respect to the time step restriction. On some manufactured solutions, numerical experiments are given to verify the predicted convergence rates of our analysis. Finally, we test the scheme on a benchmark channel flow over a step and it performs well.\\

\normalsize{\bf Keywords}: Magnetohydrodynamics; uncertainty quantification; fast ensemble calculation; finite element method; els\"asser variables

\section{\large Introduction:} When an electrically conducting fluid, e.g. plasmas, salt water and liquid metals, moves in presence of a  magnetic field, the dynamics of the magnetic field is studied in magnetohydrodynamics (MHD) and the flow is called MHD flow. Recently, the study of MHD flows has become important due to applications in e.g. engineering, physical science, geophysics and astrophysics \cite {HK09,P08,F08,DS07,BMT07,BLRY07}, liquid metal cooling of nuclear reactors \cite{BCL91,H06,SMBM10}, process metallurgy \cite{D01, SM09}, and MHD propulsion\cite{LGKP90, MG88}. The physical principle governing such flows is that the magnetic field induces currents in the moving conductive fluid, which in turn create forces on the fluid and also changes the magnetic field. The viscous, incompressible and unsteady model governed by a system of non-linear partial differential equations (PDEs) that nonlinearly couple the Navier-Stokes equations (NSEs) of fluid dynamics to the Maxwell's equations of electromagnetism, and are given in a convex domain $\Omega\subset\mathbb{R}^d(d=2\hspace{1mm}\text{or}\hspace{1mm}3)$ by \cite{B03, D01, LL60}\vspace{-2mm}
\begin{eqnarray*}
u_t+u\cdot\nabla u-B\cdot\nabla B-\nu \Delta u+\nabla p&=&f,\\
B_t+u\cdot\nabla B-B\cdot\nabla u-\nu_m\Delta B+\nabla\lambda&=&\nabla\times g,\\
\nabla\cdot u=\nabla\cdot B &= & 0,
\end{eqnarray*}
in $\Omega\times (0,T)$. Where $\Omega$ is the domain of the fluid, $u$ is velocity, $p$ is a modified pressure, $\nu$ is the kinematic viscosity, $\nu_m$ is the magnetic resistivity, $f$ is body forces, $\nabla\times g$ is the forcing on the magnetic field $B$, $T$ is the time period. The artificial magnetic pressure $\lambda$ is a Lagrange multiplier introduced in the induction equation to enforce divergence free constraint on the Maxwell equation in the discrete case but in continuous case $\lambda=0$. Assuming the domain is smooth enough, which is a common assumption in, e.g. applications in geophysics and astrophysics, we can avoid the curl formulation of the induction equation. Recently, a high order algebraic splitting method for MHD simulation was proposed in \cite{AMRX17}. 

Numerical simulations of fluid flows are greatly affected by input data like initial condition, the boundary condition, body forces, viscosity, geometry etc, which involve uncertainties. As a result uncertainty quantification (UQ) plays an important role in the validation of simulation methodologies and helps in developing rigorous methods to characterize the effect of the uncertainties on the final quantities of interest. Moreover, many fluid dynamics applications e.g. ensemble Kalman filter approach, weather forecasting, and sensitivity analyses of solutions \cite{CCDL05, GG11, LK10, LP08, L05, MX06}, require multiple numerical simulations of a flow subject to $J$ different input conditions (realizations), are then used to compute means and sensitivities. For MHD simulations, this leads to solve the following $J$ separate nonlinearly coupled systems of PDEs:
\begin{eqnarray}
u_{j,t}+u_j\cdot\nabla u_j-B_j\cdot\nabla B_j-\nu \Delta u_j+\nabla p_j &= & f_j(x,t), \hspace{2mm}\text{in}\hspace{2mm}\Omega \times (0,T), \label{gov1}\\
B_{j,t}+u_j\cdot\nabla B_j-B_j\cdot\nabla u_j-\nu_m\Delta B_j+\nabla\lambda_j &= & \nabla\times g_j(x,t)\hspace{2mm}\text{in}\hspace{2mm}\Omega \times (0,T),\label{gov3}\\
\nabla\cdot u_j & =& 0, \hspace{2mm}\text{in}\hspace{2mm}\Omega \times (0,T), \\
\nabla\cdot B_j &=& 0, \hspace{2mm}\text{in}\hspace{2mm}\Omega \times (0,T),\\ 
u_j(x,0)& =& u_j^0(x)\hspace{2mm}\text{in}\hspace{2mm}\Omega,\label{gov2}\\
B_j(x,0)& =& B_j^0(x)\hspace{2mm}\text{in}\hspace{2mm}\Omega,\label{gov5}
\end{eqnarray}
where $u_j$, $B_j$, and $p_j$ denote the solution of the $j$-{th} member of the ensemble with initial condition data $u_j^0$ and $B_j^0$, and body forces $f_j$ and $\nabla\times g_j$ and $j=1,2,\cdots, J$. For the sake of simplicity of our analysis, we consider homogeneous Dirichlet boundary conditions for both velocity and magnetic fields. For periodic boundary conditions or inhomogeneous Dirichlet boundary conditions, our analyses and results will still work after a minor modifications.  
 To obtain an accurate numerical NSE simulation for a single member of the ensemble, the required number of degrees of freedom (dof) are very high, which is known from Kolmogorov’s 1941 results \cite{L08}.
 Thus, even for a single member of MHD ensemble simulation, where velocity and magnetic field are nonlinearly coupled together, is computationally very expensive with respect to time and memory. As a result, the computational cost of the above giant system \eqref{gov1}-\eqref{gov5} will be approximately equal to $J\times$(cost of one MHD simulation) and will  generally be computationally be infeasible.  Our objective in this paper is to build and study an efficient and accurate algorithm for solving the above ensemble systems. 
 It has been shown in recent works \cite{T14,AKMR15, MR17, HMR17} that using Els\"asser variables formulation, efficient MHD simulation algorithms can be created, since they can be decoupled stable way so that at each time step, in lieu of solving a fully coupled linear system, two separate Oseen-type problems need to be solved.
  
  Defining $v_j=u_j+B_j$, $w_j=u_j-B_j$, $f_{1,j}:=f_{j}+\nabla\times g_j$, $f_{2,j}:=f_j-\nabla\times g_j$, $q_j:=p_j+\lambda_j$ and $r_j:=p_j-\lambda_j$ produces the Els{\"{a}}sser variable formulation of the ensemble systems:
\begin{eqnarray}
v_{j,t}+w_j\cdot\nabla v_j+\nabla q_j-\frac{\nu+\nu_m}{2}\Delta v_j-\frac{\nu-\nu_m}{2}\Delta w_j=f_{1,j},\label{els1}\\
w_{j,t}+v_j\cdot\nabla w_j+\nabla r_j-\frac{\nu+\nu_m}{2}\Delta w_j-\frac{\nu-\nu_m}{2}\Delta v_j=f_{2,j},\label{els2}\\
\nabla\cdot v_j=\nabla\cdot w_j=0,\label{els3}
\end{eqnarray}
together with initial and boundary conditions.

 To reduce the ensemble simulation cost, an excellent idea was presented in \cite{JL14} to find a set of $J$ solutions of the NSEs for different initial conditions and body forces. The fundamental idea is that, at each time step, each of the $J$ systems shares a common coefficient matrix but the right-hand vectors are different. Thus, the preconditioners need to build only once per time step and can reuse for all $J$ systems, also the algorithm can save storage requirement and take advantage of block linear solvers. This breakthrough idea has been implemented in heat condution\cite{F17}, Navier-Stokes simulations \cite{J15,JL15,J16,NTW16}, magnetohydrodynamics \cite{MR17}, parameterized flow problems \cite{GJW18, GJW17}, and turbulence modeling \cite{JKL15}. We use the same idea for a second oder time stepping scheme for MHD flow ensemble simulation with noisy input data. The author proposed a first order scheme to compute MHD flow ensemble subject to different initial conditions \cite{MR17} and body forces \cite{M17}.

We consider a uniform timestep size $\Delta t$ and let $t_n=n\Delta t$ for $n=0, 1, \cdots$., for simplicity, we suppress the spatial discretization momentarily. Then computing the $J$ solutions independently, takes the following form:\\ 
Step 1: for $j$=1,...,$J$, 
\begin{eqnarray}
\frac{3v_j^{n+1}-4v_j^n+v_j^{n-1}}{2\Delta t}+\nabla q_j^{n+1}-\frac{\nu+\nu_m}{2}\Delta v_j^{n+1}-\frac{\nu-\nu_m}{2}\Delta(2w_j^n-w_j^{n-1})-\frac{\nu-\nu_m}{2}\Delta w_j^n\nonumber\\ +<w>^n\cdot\nabla v_j^{n+1}+w_j^{'n}\cdot\nabla(2v_j^n-v_j^{n-1})  = f_{1,j}(t^{n+1}),\hspace{4mm} \nabla\cdot v_j^{n+1}=0\label{scheme1}
\end{eqnarray}
Step 2: for $j$=1,...,$J$,
\begin{eqnarray}
\frac{3w_j^{n+1}-4w_j^n+w_j^{n-1}}{2\Delta t}+\nabla r_j^{n+1}-\frac{\nu+\nu_m}{2}\Delta w_j^{n+1}-\frac{\nu-\nu_m}{2}\Delta(2v_j^n-v_j^{n-1})-\frac{\nu-\nu_m}{2}\Delta v_j^n\nonumber\\ +<v>^n\cdot\nabla w_j^{n+1}+v_j^{{'}n}\cdot\nabla(2 w_j^n-w_j^{n-1})=f_{2,j}(t^{n+1}),\hspace{4mm} \nabla\cdot w_j^{n+1}=0\label{scheme2}
\end{eqnarray}
where $v_j^n,w_j^{n+1}, q_j^n$ and $r_j^{n+1}$ denote approximations of $v_j(\cdot,t^n),w_j(\cdot,t^n), q_j(\cdot,t^n)$ and $r_j(\cdot,t^n)$ in \eqref{els1}-\eqref{els3}.  \textcolor{black}{The ensemble mean and fluctuation about the mean are denoted by $<u>$, $u_j^{'}$ respectively and these are defined as follows: 
\begin{eqnarray}
<u>^n:=\frac{1}{J}\sum\limits_{j=1}^{J}(2u_j^n-u_j^{n-1}), \hspace{2mm} u_j^{'n}:=2u_j^n-u_{j}^{n-1}-<u>^n.
\end{eqnarray}
}
The key to the efficiencies of the above algorithm are that (1) the MHD system is decoupled into two Oseen problems and can be solved simultaneously if the computational resources are available, (2) the coefficient matrices of \eqref{scheme1} and \eqref{scheme2} at each time step are independent of $j$, thus all the $J$ members for each sub-problems in the ensemble share a same coefficient matrix. That is, at every time step, we do not need to solve $J$ individual systems of equations for each sub-problem instead a single linear system with $J$ different right-hand-side constant vectors.    

We give a rigorous proof that the decoupled scheme is stable and the ensemble of $J$ computed solutions converges to the ensemble solution of the $J$ true MHD solutions, as the timestep size and the spatial mesh width tend to zero.

This paper is organized as follows.  In section 2, we give notation and mathematical preliminaries that will allow for a smooth presentation and analysis to follow.  Section 3 presents and analyzes a fully discrete algorithm corresponding to \eqref{scheme1}-\eqref{scheme2}, and proves it is stable and convergent.  Numerical tests are presented in section 4, and finally conclusions are drawn in section 5.

\section{\large Notation and Preliminaries:}

Let $\Omega\subset \mathbb{R}^d\ (d=2,3)$ be a convex polygonal or polyhedral domain in $\mathbb{R}^d(d=2,3)$ with boundary $\partial\Omega$. The usual $L^2(\Omega)$ norm and inner product are denoted by $\|.\|$ and $(.,.)$ respectively. Similarly, the $L^p(\Omega)$ norms and the Sobolev $W_p^k(\Omega)$ norms are $\|.\|_{L^p}$ and $\|.\|_{W_p^k}$ respectively for $k\in\mathbb{N},\hspace{1mm}1\le p\le \infty$. Sobolev space $W_2^k(\Omega)$ is represented by $H^k(\Omega)$ with norm $\|.\|_k$. The natural function spaces for our problem are
$$X:=H_0^1(\Omega)=\{v\in (L^p(\Omega))^d :\nabla v\in L^2(\Omega)^{d\times d}, v=0 \hspace{2mm} \mbox{on}\hspace{2mm}   \partial \Omega\},$$
$$Q:=L_0^2(\Omega)=\{ q\in L^2(\Omega): \int_\Omega q\hspace{2mm}dx=0\}.$$
Recall the Poincare inequality holds in $X$: there exists $C$ depending only on the size of $\Omega$ satisfying for all $\phi\in X$,
\[
\| \phi \| \le C \| \nabla \phi \|.
\]
The divergence free velocity space is given by
$$V:=\{v\in X:(\nabla\cdot v, q)=0, \forall q\in Q\}.$$
We define the trilinear form $b:X\times X\times X\rightarrow \mathbb{R}$ by
 \[
 b(u,v,w):=(u\cdot\nabla v,w), 
 \]
and recall from \cite{GR86} that $b(u,v,v)=0$ if $u\in V$, and 
\begin{align}
|b(u,v,w)|\leq C(\Omega)\|\nabla u\|\|\nabla v\|\|\nabla w\|,\hspace{2mm}\mbox{for any}\hspace{2mm}u,v,w\in X.\label{nonlinearbound}
\end{align}
The conforming finite element spaces are denoted by $X_h\subset X$ and  $Q_h\subset Q$, and we assume a regular triangulation $\tau_h(\Omega)$, where the maximum triangle diameter.   We assume that $(X_h,Q_h)$ satisfies the usual discrete inf-sup condition
\begin{eqnarray}
\inf_{q_h\in Q_h}\sup_{v_h\in X_h}\frac{(q_h,\grad\cdot v_h)}{\|q_h\|\|\grad v_h\|}\geq\beta>0,\label{infsup}
\end{eqnarray}
where $\beta$ is independent of $h$.

The space of discretely divergence free functions is defined as $$V_h:=\{v_h\in X_h:(\nabla\cdot v_h,q_h)=0,\hspace{2mm}\forall\hspace{2mm}q_h\in Q_h\}.$$ 

We use the $(X_h,Q_h, X_h, Q_h)=(P_k,P_{k-1}, P_k,P_{k-1} )$ Taylor-Hood (TH) finite element pair for both our analysis and computations, which satisfies the inf-sup condition for the polynomial degree $k\ge d$ \cite{arnold:qin:scott:vogelius:2D,Z05}. Our analysis can be extended without difficulty to any inf-sup stable element choice, {\color{black} however, there will be additional terms that appear in the convergence analysis if non-divergence-free elements are chosen.  In particular, pressure robustness of the convergence estimates will be lost, as the error will be dependent on the size of true solution pressure derivatives.}

%
%

We have the following approximation properties in $(X_h,Q_h)$: \cite{BS08}
\begin{align}
\inf_{v_h\in X_h}\|u-v_h\|&\leq Ch^{k+1}|u|_{k+1},\hspace{2mm}u\in H^{k+1}(\Omega),\label{AppPro1}\\
 \inf_{v_h\in X_h}\|\grad (u-v_h)\|&\leq Ch^{k}|u|_{k+1},\hspace{5mm}u\in H^{k+1}(\Omega),\label{AppPro2}\\
\inf_{q_h\in Q_h}\|p-q_h\|&\leq Ch^k|p|_k,\hspace{10mm}p\in H^k(\Omega),
\end{align}
where $|\cdot|_r$ denotes the $H^r$ seminorm.

We will assume the mesh is sufficiently regular for the inverse inequality to hold, and with this and the LBB assumption, we have approximation properties
\begin{align}
\| \nabla ( u- P_{L^2}^{V_h}(u)  ) \|&\leq Ch^{k}|u|_{k+1},\hspace{2mm}u\in H^{k+1}(\Omega),\label{AppPro3}\\
 \inf_{v_h\in V_h}\|\grad (u-v_h)\|&\leq Ch^{k}|u|_{k+1},\hspace{2mm}u\in H^{k+1}(\Omega),\label{AppPro4}
\end{align}
where $P_{L^2}^{V_h}(u)$ is the $L^2$ projection of $u$ into $V_h$.

\textcolor{black}{The following lemma for the discrete Gronwall inequality was given in \cite{HR90}.
\begin{lemma}
Let $\Delta t$, $H$, $a_n$, $b_n$, $c_n$, $d_n$ be non-negative numbers for $n=1,\cdots, M$ such that
    $$a_M+\Delta t \sum_{n=1}^Mb_n\leq \Delta t\sum_{n=1}^{M-1}{d_na_n}+\Delta t\sum_{n=1}^Mc_n+H\hspace{3mm}\mbox{for}\hspace{2mm}M\in\mathbb{N},$$
then for all $\Delta t> 0,$
$$a_M+\Delta t\sum_{n=1}^Mb_n\leq \mbox{exp}\left(\Delta t\sum_{n=1}^{M-1}d_n\right)\left(\Delta t\sum_{n=1}^Mc_n+H\right)\hspace{2mm}\mbox{for}\hspace{2mm}M\in\mathbb{N}.$$
\end{lemma}}

\section{\large Fully discrete scheme and analysis of ensemble eddy viscosity:}

We are now ready to present the fully discrete scheme for efficient MHD ensemble calculations.  It equips 
\eqref{els1}-\eqref{els3} with a finite element spatial discretization.  The scheme is defined as follows.
\begin{algorithm}\label{Algn1}
Given time step $\Delta t>0$, end time $T>0$, initial conditions $v_j^0, w_j^0, v_j^1, w_j^1\in V_h$ and $f_{1,j}, f_{2,j}\in L^\infty(0,T;H^{-1}(\Omega)^d)$ for $j=1,2,\cdots, J$. Set $M=T/\Delta t$ and for $n=1,\cdots, M-1$, compute: \\\\
Find $v_{j,h}^{n+1}\in V_h$ satisfying, for all $\chi_h\in V_h$ :
\begin{align}
\Bigg(\frac{3v_{j,h}^{n+1}-4v_{j,h}^n+v_{j,h}^{n-1}}{2\Delta t},\chi_h\Bigg)+\frac{\nu+\nu_m}{2}(\nabla v_{j,h}^{n+1},\nabla \chi_h)+\frac{\nu-\nu_m}{2}(\nabla (2w_{j,h}^n-w_{j,h}^{n-1}),\nabla \chi_h)\nonumber\\+(<w_h>^n\cdot\nabla v_{j,h}^{n+1},\chi_h)+(w_{j,h}^{'n}\cdot\nabla (2v_{j,h}^n-v_{j,h}^{n-1}),\chi_h) = (f_{1,j}(t^{n+1}),\chi_h),\label{weakn1}
\end{align}
Find $w_{j,h}^{n+1}\in V_h$ satisfying, for all $l_h\in V_h$ :

\begin{eqnarray}
\left(\frac{3w_{j,h}^{n+1}-4w_{j,h}^n+w_{j,h}^{n-1}}{\Delta t},l_h\right)+\frac{\nu+\nu_m}{2}(\nabla w_{j,h}^{n+1},\nabla l_h)+\frac{\nu-\nu_m}{2}(\nabla (2v_{j,h}^n-v_{j,h}^{n-1}),\nabla l_h)\nonumber\\+(<v_h>^n\cdot\nabla w_{j,h}^{n+1},l_h)+(v_{j,h}^{'n}\cdot\nabla(2w_{j,h}^n-w_{j,h}^{n-1}),l_h)=(f_{2,j}(t^{n+1}),l_h).\label{weakn2}
\end{eqnarray}
\end{algorithm}

\subsection{\large Stability Analysis:}
We now prove stability and well-posedness for the Algorithm \eqref{Algn1}. To simplify our calculation, we denote $\alpha:=\nu+\nu_m-|\nu-\nu_m|>0.$ 
\begin{lemma}
Consider the Algorithm \ref{Algn1}. If the mesh is sufficiently regular so that the inverse inequality holds (with constant $C_i$) and the time step is chosen to satisfy $$\Delta t\le \frac{\alpha h^2}{3(\nu-\nu_m)^2C_i+12C^2C_i^2\max\limits_{1\le j\le J}\big\{\|\nabla v_{j,h}^{'n}\|,\|\nabla v_{j,h}^{'n}\|\big\}}$$
then the method is stable and solutions to \eqref{weakn1}-\eqref{weakn2} satisfy

\begin{align}
\|v_{j,h}^M\|^2+\|w_{j,h}^M\|^2+\|2v_{j,h}^M&-v_{j,h}^{M-1}\|^2+\|2w_{j,h}^M-w_{j,h}^{M-1}\|^2+\alpha\Delta t\sum\limits_{n=1}^{M-1}(\|\nabla v_{j,h}^{n+1}\|^2+\|\nabla w_{j,h}^{n+1}\|^2)\nonumber\\&\le \|v_{j,h}^1\|^2+\|w_{j,h}^1\|^2+\|2v_{j,h}^1-v_{j,h}^0\|^2+\|2w_{j,h}^1-w_{j,h}^0\|^2\nonumber\\&+\frac{12\Delta t}{\alpha}\sum\limits_{n=1}^{M-1}(\|f_{1,j}(t^{n+1})\|_{-1}+\|f_{2,j}(t^{n+1})\|_{-1}).\label{stability1}
\end{align}
\end{lemma}

\begin{proof}
Choos $\chi_h=v_{j,h}^{n+1}$ in \eqref{weakn1}, using the following identity
\begin{eqnarray}
(3a-4b+c,a)=\frac{a^2+(2a-b)^2}{2}-\frac{b^2+(2b-c)^2}{2}+\frac{(a-2b+c)^2}{2},\label{ident}
\end{eqnarray}
we obtain
\begin{align}
\frac{1}{4\Delta t}\bigg(\|v_{j,h}^{n+1}\|^2-\|v_{j,h}^n\|^2+\|2v_{j,h}^{n+1}-v_{j,h}^n\|^2-\|2v_{j,h}^n-v_{j,h}^{n-1}\|^2+\|v_{j,h}^{n+1}-2v_{j,h}^n+v_{j,h}^{n-1}\|^2\bigg)\nonumber\\+\frac{\nu+\nu_m}{2}\|\nabla v_{j,h}^{n+1}\|^2+\frac{\nu-\nu_m}{2}\big(\nabla(2 w_{j,h}^n- w_{j,h}^{n-1}),\nabla v_{j,h}^{n+1}\big)+(w_{j,h}^{'n}\cdot\nabla(2v_{j,h}^n-v_{j,h}^{n-1}),v_{j,h}^{n+1})\nonumber\\=(f_{1,j}(t^{n+1}),v_{j,h}^{n+1}).\label{MHW1}
\end{align}
Similarly, choose $l_h=w_{j,h}^{n+1}$ in \eqref{weakn2}, we have

\begin{align}
&\frac{1}{4\Delta t}\bigg(\|w_{j,h}^{n+1}\|^2-\|w_{j,h}^n\|^2+\|2w_{j,h}^{n+1}-w_{j,h}^n\|^2-\|2w_{j,h}^n-w_{j,h}^{n-1}\|^2+\|w_{j,h}^{n+1}-2w_{j,h}^n+w_{j,h}^{n-1}\|^2\bigg)\nonumber\\&+\frac{\nu+\nu_m}{2}\|\nabla w_{j,h}^{n+1}\|^2+\frac{\nu-\nu_m}{2}\big(\nabla(2 v_{j,h}^n- v_{j,h}^{n-1}),\nabla w_{j,h}^{n+1}\big)+(v_{j,h}^{'n}\cdot\nabla(2w_{j,h}^n-w_{j,h}^{n-1}),w_{j,h}^{n+1})\nonumber\\&=(f_{2,j}(t^{n+1}),w_{j,h}^{n+1}).\label{MHW2}
\end{align}
Next, using
\begin{align*}
(w_{j,h}^{'n}&\cdot\nabla(2v_{j,h}^n-v_{j,h}^{n-1}),v_{j,h}^{n+1})=(w_{j,h}^{'n}\cdot\nabla v_{j,h}^{n+1},v_{j,h}^{n+1}-2v_{j,h}^n+v_{j,h}^{n-1})\\&\le C\|\nabla w_{j,h}^{'n}\|\|\nabla v_{j,h}^{n+1}\|\hspace{1mm}\|\nabla (v_{j,h}^{n+1}-2v_{j,h}^n+v_{j,h}^{n-1})\|\\&\le \frac{CC_i}{h} \|\nabla w_{j,h}^{'n}\|\|\nabla v_{j,h}^{n+1}\|\hspace{1mm}\|v_{j,h}^{n+1}-2v_{j,h}^n+v_{j,h}^{n-1}\|,
\end{align*}
adding equations \eqref{MHW1} and \eqref{MHW2} and applying Cauchy-Schwarz inequality, yields
\begin{align*}
\frac{1}{4\Delta t}\bigg(\|v_{j,h}^{n+1}\|^2-\|v_{j,h}^n\|^2+\|2v_{j,h}^{n+1}-v_{j,h}^n\|^2-\|2v_{j,h}^n-v_{j,h}^{n-1}\|^2+\|v_{j,h}^{n+1}-2v_{j,h}^n+v_{j,h}^{n-1}\|^2\\+\|w_{j,h}^{n+1}\|^2-\|w_{j,h}^n\|^2+\|2w_{j,h}^{n+1}-w_{j,h}^n\|^2-\|2w_{j,h}^n-w_{j,h}^{n-1}\|^2+\|w_{j,h}^{n+1}-2w_{j,h}^n+w_{j,h}^{n-1}\|^2\bigg)\\+\frac{\nu+\nu_m}{2}\big(\|\nabla v_{j,h}^{n+1}\|^2+\|\nabla w_{j,h}^{n+1}\|^2\big)\\+\frac{\nu-\nu_m}{2}\bigg\{\big(\nabla(2v_{j,h}^n-v_{j,h}^{n-1}),\nabla w_{j,h}^{n+1}\big)+\big(\nabla(2w_{j,h}^n-w_{j,h}^{n-1}),\nabla v_{j,h}^{n+1}\big)\bigg\}\\\le \frac{CC_i}{h}\|\nabla w_{j,h}^{'n}\|\|\nabla v_{j,h}^{n+1}\|\hspace{1mm}\|v_{j,h}^{n+1}-2v_{j,h}^n+v_{j,h}^{n-1}\|+\frac{CC_i}{h}\|\nabla v_{j,h}^{'n}\|\|\nabla w_{j,h}^{n+1}\|\hspace{1mm}\|w_{j,h}^{n+1}-2w_{j,h}^n+w_{j,h}^{n-1}\|\\+\|f_{1,j}(t^{n+1})\|_{-1}\|\nabla v_{j,h}^{n+1}\|+\|f_{2,j}(t^{n+1})\|_{-1}\|\nabla w_{j,h}^{n+1}\|.
\end{align*}
Adding and subtracting the term $\frac{\nu-\nu_m}{2}\left(\nabla v_{j,h}^{n+1},\nabla w_{j,h}^{n+1}\right)$ twice provides
\begin{align*}
\frac{1}{4\Delta t}\bigg(\|v_{j,h}^{n+1}\|^2-\|v_{j,h}^n\|^2+\|2v_{j,h}^{n+1}-v_{j,h}^n\|^2-\|2v_{j,h}^n-v_{j,h}^{n-1}\|^2+\|v_{j,h}^{n+1}-2v_{j,h}^n+v_{j,h}^{n-1}\|^2\\+\|w_{j,h}^{n+1}\|^2-\|w_{j,h}^n\|^2+\|2w_{j,h}^{n+1}-w_{j,h}^n\|^2-\|2w_{j,h}^n-w_{j,h}^{n-1}\|^2+\|w_{j,h}^{n+1}-2w_{j,h}^n+w_{j,h}^{n-1}\|^2\bigg)\\+\frac{\nu+\nu_m}{2}\big(\|\nabla v_{j,h}^{n+1}\|^2+\|\nabla w_{j,h}^{n+1}\|^2\big)\\-\frac{\nu-\nu_m}{2}\bigg\{\big(\nabla(v_{j,h}^{n+1}-2v_{j,h}^n+v_{j,h}^{n-1}),\nabla w_{j,h}^{n+1}\big)+\big(\nabla(w_{j,h}^{n+1}-2w_{j,h}^n+w_{j,h}^{n-1}),\nabla v_{j,h}^{n+1}\big)\bigg\}\\+\frac{\nu-\nu_m}{2}(\nabla v_{j,h}^{n+1},\nabla w_{j,h}^{n+1})+\frac{\nu-\nu_m}{2}(\nabla w_{j,h}^{n+1},\nabla v_{j,h}^{n+1})\\\le \frac{CC_i}{h}\|\nabla w_{j,h}^{'n}\|\|\nabla v_{j,h}^{n+1}\|\hspace{1mm}\|v_{j,h}^{n+1}-2v_{j,h}^n+v_{j,h}^{n-1}\|+\frac{CC_i}{h}\|\nabla v_{j,h}^{'n}\|\|\nabla w_{j,h}^{n+1}\|\hspace{1mm}\|w_{j,h}^{n+1}-2w_{j,h}^n+w_{j,h}^{n-1}\|\\+\|f_{1,j}(t^{n+1})\|_{-1}\|\nabla v_{j,h}^{n+1}\|+\|f_{2,j}(t^{n+1})\|_{-1}\|\nabla w_{j,h}^{n+1}\|.
\end{align*}
Using Cauchy-Schwarz and Young's inequalities we have that
\begin{align}
&\frac{1}{4\Delta t}\bigg(\|v_{j,h}^{n+1}\|^2-\|v_{j,h}^n\|^2+\|2v_{j,h}^{n+1}-v_{j,h}^n\|^2-\|2v_{j,h}^n-v_{j,h}^{n-1}\|^2+\|v_{j,h}^{n+1}-2v_{j,h}^n+v_{j,h}^{n-1}\|^2\nonumber\\&+\|w_{j,h}^{n+1}\|^2-\|w_{j,h}^n\|^2+\|2w_{j,h}^{n+1}-w_{j,h}^n\|^2-\|2w_{j,h}^n-w_{j,h}^{n-1}\|^2+\|w_{j,h}^{n+1}-2w_{j,h}^n+w_{j,h}^{n-1}\|^2\bigg)\nonumber\\&+\frac{\nu+\nu_m}{2}\big(\|\nabla v_{j,h}^{n+1}\|^2+\|\nabla w_{j,h}^{n+1}\|^2\big)\le|\nu-\nu_m|\|\nabla v_{j,h}^{n+1}\|\|\nabla w_{j,h}^{n+1}\|\nonumber\\&+\frac{|\nu-\nu_m|}{2}\|\nabla (v_{j,h}^{n+1}-2v_{j,h}^n+v_{j,h}^{n-1})\|\|\nabla w_{j,h}^{n+1}\|+\frac{|\nu-\nu_m|}{2}\|\nabla (w_{j,h}^{n+1}-2w_{j,h}^n+w_{j,h}^{n-1})\|\|\nabla v_{j,h}^{n+1}\|\nonumber\\&+\frac{CC_i}{h}\|\nabla w_{j,h}^{'n}\|\|\nabla v_{j,h}^{n+1}\|\hspace{1mm}\|v_{j,h}^{n+1}-2v_{j,h}^n+v_{j,h}^{n-1}\|+\frac{CC_i}{h}\|\nabla v_{j,h}^{'n}\|\|\nabla w_{j,h}^{n+1}\|\hspace{1mm}\|w_{j,h}^{n+1}-2w_{j,h}^n+w_{j,h}^{n-1}\|\nonumber\\&+\|f_{1,j}(t^{n+1})\|_{-1}\|\nabla v_{j,h}^{n+1}\|+\|f_{2,j}(t^{n+1})\|_{-1}\|\nabla w_{j,h}^{n+1}\|.\label{bounds1}
\end{align}
Young's inequality provides the following bounds on the last seven terms in \eqref{bounds1}:
\begin{align*}
|\nu-\nu_m|\hspace{1mm}\|\nabla v_{j,h}^{n+1}\|\|\nabla w_{j,h}^{n+1}\|&\le\frac{|\nu-\nu_m|}{2}\big(\|\nabla v_{j,h}^{n+1}\|^2+\|\nabla w_{j,h}^{n+1}\|^2\big),\\
\frac{|\nu-\nu_m|}{2}\|\nabla (v_{j,h}^{n+1}-& 2v_{j,h}^n + v_{j,h}^{n-1})\|\|\nabla w_{j,h}^{n+1}\|\\&\le \frac{\alpha}{12}\|\nabla w_{j,h}^{n+1}\|^2+\frac{3(\nu-\nu_m)^2}{4\alpha}\|\nabla (v_{j,h}^{n+1}-2v_{j,h}^n+v_{j,h}^{n-1})\|^2,\\
\frac{|\nu-\nu_m|}{2}\|\nabla (w_{j,h}^{n+1}-& 2w_{j,h}^n+w_{j,h}^{n-1})\|\|\nabla v_{j,h}^{n+1}\|\\&\le \frac{\alpha}{12}\|\nabla v_{j,h}^{n+1}\|^2+\frac{3(\nu-\nu_m)^2}{4\alpha}\|\nabla (w_{j,h}^{n+1}-2w_{j,h}^n+w_{j,h}^{n-1})\|^2,\\
\frac{CC_i}{h}\|\nabla w_{j,h}^{'n}\|\|\nabla v_{j,h}^{n+1}\|\hspace{1mm}\|& v_{j,h}^{n+1}-2v_{j,h}^n+v_{j,h}^{n-1}\|\\&\le \frac{\alpha}{12}\|\nabla v_{j,h}^{n+1}\|^2+\frac{3C^2C_i^2\|\nabla w_{j,h}^{'n}\|^2}{\alpha h^2}\|v_{j,h}^{n+1}-2v_{j,h}^n+v_{j,h}^{n-1}\|^2,\\
\frac{CC_i}{h}\|\nabla v_{j,h}^{'n}\|\|\nabla w_{j,h}^{n+1}\|\hspace{1mm}\|& w_{j,h}^{n+1}-2w_{j,h}^n+w_{j,h}^{n-1}\|\\&\le \frac{\alpha}{12}\|\nabla w_{j,h}^{n+1}\|^2+\frac{3C^2C_i^2\|\nabla v_{j,h}^{'n}\|^2}{\alpha h^2}\|w_{j,h}^{n+1}-2w_{j,h}^n+w_{j,h}^{n-1}\|^2,\\
\|f_{1,j}(t^{n+1})\|_{-1}\|\nabla v_{j,h}^{n+1}\|&\le \frac{\alpha}{12}\|\nabla v_{j,h}^{n+1}\|^2+\frac{3}{\alpha}\|f_{1,j}(t^{n+1})\|_{-1}^2,\\
\|f_{2,j}(t^{n+1})\|_{-1}\|\nabla w_{j,h}^{n+1}\|&\le \frac{\alpha}{12}\|\nabla w_{j,h}^{n+1}\|^2+\frac{3}{\alpha}\|f_{2,j}(t^{n+1})\|_{-1}^2.
\end{align*}

Using these estimates and the following inverse inequality in \eqref{bounds1} 
$$\|\nabla (u_{j,h}^{n+1}-2u_{j,h}^n+u_{j,h}^{n-1})\|^2\le \frac{C_i^2}{h^2}\|u_{j,h}^{n+1}-2u_{j,h}^n+u_{j,h}^{n-1}\|^2$$
produces

\begin{align}
\frac{1}{4\Delta t}\bigg(\|v_{j,h}^{n+1}\|^2-\|v_{j,h}^n\|^2+\|2v_{j,h}^{n+1}-v_{j,h}^n\|^2-\|2v_{j,h}^n-v_{j,h}^{n-1}\|^2\nonumber\\+\|w_{j,h}^{n+1}\|^2-\|w_{j,h}^n\|^2+\|2w_{j,h}^{n+1}-w_{j,h}^n\|^2-\|2w_{j,h}^n-w_{j,h}^{n-1}\|^2\bigg)\nonumber\\+\bigg\{\frac{1}{4\Delta t}-\frac{3(\nu-\nu_m)^2C_i+12C^2C_i^2\|\nabla w_{j,h}^{'n}\|^2}{4\alpha h^2}\bigg\}\|v_{j,h}^{n+1}-2v_{j,h}^n+v_{j,h}^{n-1}\|^2\nonumber\\+\bigg\{\frac{1}{4\Delta t}-\frac{3(\nu-\nu_m)^2C_i+12C^2C_i^2\|\nabla v_{j,h}^{'n}\|^2}{4\alpha h^2}\bigg\}\|w_{j,h}^{n+1}-2w_{j,h}^n+w_{j,h}^{n-1}\|^2\nonumber\\+\frac{\alpha}{4}\big(\|\nabla v_{j,h}^{n+1}\|^2+\|\nabla w_{j,h}^{n+1}\|^2\big)\le \frac{3}{\alpha}(\|f_{1,j}(t^{n+1})\|_{-1}+\|f_{2,j}(t^{n+1})\|_{-1}).\label{lastbound}
\end{align}

Now if we choose $\Delta t\le \frac{\alpha h^2}{3(\nu-\nu_m)^2C_i+12C^2C_i^2\max\limits_{1\le j\le J}\big\{\|\nabla v_{j,h}^{'n}\|,\|\nabla v_{j,h}^{'n}\|\big\}}$, dropping the non-negative terms on left, multiplying both sides by $4\Delta t$ and summing over time steps from $n=1$ to $n=M-1$ results in \eqref{stability1}.
\end{proof}

\subsection{\large Error Analysis:}\label{ErrorAnalysis}
Now we consider the convergence of the proposed decoupled scheme.
\begin{theorem}
For $(v_j,w_j,q_j,r_j)$ satisfying \eqref{els1}-\eqref{els3} with regularity assumptions $v_j$,$w_j$ $\in L^{\infty}(0, T;$ $H^{m}(\Omega)^d)$ for $m=\max\{2,k+1\}$, $v_{j,tt}, w_{j,tt}\in L^{\infty}(0,T;H^1(\Omega)^d)$ and $v_{j,ttt}, w_{j,ttt}\in L^{\infty}(0,T;L^2(\Omega)^d)$. Then the ensemble solution $(<v_{h}>, <w_{h}>)$ to Algorithm \eqref{Algn1} converges to the true ensemble solution: for any $\Delta t\le \frac{\alpha h^2}{9C_i^2(\nu-\nu_m)^2+9C_i^2C^2\max\limits_{1\le j\le J}\big\{\|\nabla v_{j,h}^{'n}\|,\|\nabla w_{j,h}^{'n}\|\big\}}$, one has
\begin{align}
\|<v>^T-<v_h>^M\|^2+&\alpha\Delta t\sum\limits_{n=2}^{M}\|\nabla(<v>(t^n)-<v_h>^n)\|^2\nonumber\\&\le \frac{2^{J+2}C}{J^2\alpha} e^{\frac{9TC}{\alpha}}((\nu^2+\nu_m^2+1)h^{2k}+((\nu-\nu_m)^2+1)\Delta t^4)
\end{align}
\end{theorem}
\begin{proof}
We start our proof by obtaining the error equation. Testing \eqref{els1} and \eqref{els2} with $\chi_h, l_h\in V_h$ at the time level $t^{n+1}$, the continuous variational formulations can be written as
\begin{align}
\bigg(\frac{3v_j(t^{n+1})-4v_j(t^n)+v_j(t^{n-1})}{2\Delta t},\chi_h\bigg)+(w_j(t^{n+1})\cdot\nabla v_j(t^{n+1}),\chi_h)\nonumber\\+\frac{\nu+\nu_m}{2}(\nabla v_j(t^{n+1}), \nabla\chi_h) +\frac{\nu-\nu_m}{2}(\nabla(2w_j(t^n)-w_j(t^{n-1})),\chi_h)-(q_j(t^{n+1})-\rho_{j,h},\nabla\cdot\chi_h)\nonumber\\ =(f_{1,j}(t^{n+1}),\chi_h)-\frac{\nu-\nu_m}{2}\big(\nabla( w_j(t^{n+1})-2w_j(t^n)+w_j(t^{n-1})),\chi_h\big)\nonumber\\-\bigg(v_{j,t}(t^{n+1})-\frac{3v_j(t^{n+1})-4v_j(t^n)+v_j(t^{n-1})}{2\Delta t}, \chi_h\bigg), \label{conweakn1}
\end{align}
and
\begin{align}
\bigg(\frac{3w_j(t^{n+1})-4w_j(t^n)+w_j(t^{n-1})}{2\Delta t},l_h\bigg)+(v_j(t^{n+1})\cdot\nabla w_j(t^{n+1}),l_h)\nonumber\\+\frac{\nu+\nu_m}{2}(\nabla w_j(t^{n+1}), \nabla l_h) +\frac{\nu-\nu_m}{2}(\nabla(2v_j(t^n)-v_j(t^{n-1})),l_h)-(r_j(t^{n+1})-\zeta_{j,h},\nabla\cdot l_h)\nonumber\\ =(f_{2,j}(t^{n+1}),l_h)-\frac{\nu-\nu_m}{2}\big(\nabla( v_j(t^{n+1})-2v_j(t^n)+v_j(t^{n-1})),l_h\big)\nonumber\\-\bigg(w_{j,t}(t^{n+1})-\frac{3w_j(t^{n+1})-4w_j(t^n)+w_j(t^{n-1})}{2\Delta t}, l_h\bigg). \label{conweakn2}
\end{align}
Denote $e_{v,j}^n:=v_j(t^n)-v_{j,h}^n,\hspace{2mm}e_{w,j}^n:=w_j(t^n)-w_{j,h}^n.$ Subtracting \eqref{weakn1} and \eqref{weakn2} from equation \eqref{conweakn1} and \eqref{conweakn2} respectively, yields 

\begin{align}
\bigg(&\frac{3e_{j,v}^{n+1}-4e_{j,v}^n+e_{j,v}^{n-1}}{2\Delta t},\chi_h\bigg)+\frac{\nu+\nu_m}{2}(\nabla e_{j,v}^{n+1},\nabla \chi_h)+\frac{\nu-\nu_m}{2}\big(\nabla (2e_{j,w}^n-e_{j,w}^{n-1}),\nabla\chi_h\big)\nonumber\\&+((2e_{j,w}^n-e_{j,w}^{n-1})\cdot\nabla v_j(t^{n+1}),\chi_h)+((2w_{j,h}^n-w_{j,h}^{n-1})\cdot\nabla e_{j,v}^{n+1},\chi_h)\nonumber\\&-(w_{j,h}^{'n}\cdot\nabla(e_{j,v}^{n+1}-2e_{j,v}^n+e_{j,v}^{n-1}),\chi_h)=-G_1(t,v_j,w_j,\chi_h),
\end{align}
and
\begin{align}
\bigg(&\frac{3e_{j,w}^{n+1}-4e_{j,w}^n+e_{j,w}^{n-1}}{2\Delta t},l_h\bigg)+\frac{\nu+\nu_m}{2}(\nabla e_{j,w}^{n+1},\nabla l_h)+\frac{\nu-\nu_m}{2}\big(\nabla (2e_{j,v}^n-e_{j,v}^{n-1}),\nabla l_h\big)\nonumber\\&+((2e_{j,v}^n-e_{j,v}^{n-1})\cdot\nabla w_j(t^{n+1}),l_h)+((2v_{j,h}^n-v_{j,h}^{n-1})\cdot\nabla e_{j,w}^{n+1},l_h)\nonumber\\&-(v_{j,h}^{'n}\cdot\nabla(e_{j,w}^{n+1}-2e_{j,w}^n+e_{j,w}^{n-1}),l_h)=-G_2(t,v_j,w_j,l_h),
\end{align}
where
\begin{align*}
G_1(t,v_j,w_j,\chi_h):=((w_j(t^{n+1})-2w_j(t^n)+w_j(t^{n-1}))\cdot\nabla v_j(t^{n+1}),\chi_h)\nonumber\\+(w_{j,h}^{'n}\cdot\nabla(v_j(t^{n+1})-2v_j(t^n)+v_j(t^{n-1})),\chi_h)\nonumber\\+\frac{\nu-\nu_m}{2}\big(\nabla( w_j(t^{n+1})-2w_j(t^n)+w_j(t^{n-1})),\nabla\chi_h\big)\nonumber\\+\bigg(v_{j,t}(t^{n+1})-\frac{3v_j(t^{n+1})-4v_j(t^n)+v_j(t^{n-1})}{2\Delta t}, \chi_h\bigg)
\end{align*}
and
\begin{align*}
G_2(t,v_j,w_j,l_h):=((v_j(t^{n+1})-2v_j(t^n)+v_j(t^{n-1}))\cdot\nabla w_j(t^{n+1}),l_h)\nonumber\\+(v_{j,h}^{'n}\cdot\nabla(w_j(t^{n+1})-2w_j(t^n)+w_j(t^{n-1})),l_h)\nonumber\\+\frac{\nu-\nu_m}{2}\big(\nabla( v_j(t^{n+1})-2v_j(t^n)+v_j(t^{n-1})),\nabla l_h\big)\nonumber\\+\bigg(w_{j,t}(t^{n+1})-\frac{3w_j(t^{n+1})-4w_j(t^n)+w_j(t^{n-1})}{2\Delta t}, l_h\bigg).
\end{align*}
Now we decompose the errors as
$$e_{j,v}^n: = v_j(t^n)-v_{j,h}^n=(v_j(t^n)-\tilde{v}_j^n)-(v_{j,h}^n-\tilde{v}_j^n):=\eta_{j,v}^n-\phi_{j,h}^n,$$
$$e_{j,w}^n: = w_j(t^n)-w_{j,h}^n=(w_j(t^n)-\tilde{w}_j^n)-(w_{j,h}^n-\tilde{w}_j^n):=\eta_{j,w}^n-\psi_{j,h}^n,$$

where $\tilde{v}_j^n: =P_{V_h}^{L^2}(v_j(t^n))\in V_h$ and $\tilde{w}_j^n: =P_{V_h}^{L^2}(w_j(t^n))\in V_h$ are the $L^2$ projections of $v_j(t^n)$ and $w_j(t^n)$ into $V_h$, respectively. Note that $(\eta_{j,v}^n,v_h)=(\eta_{j,w}^n,v_h)=0\hspace{2mm} \forall v_h\in V_h.$  Rewriting, we have for $\chi_h, l_h\in V_h$
\begin{align}
\bigg(\frac{3\phi_{j,h}^{n+1}-4\phi_{j,h}^n+\phi_{j,h}^{n-1}}{2\Delta t},\chi_h\bigg)+\frac{\nu+\nu_m}{2}(\nabla \phi_{j,h}^{n+1},\nabla \chi_h)+\frac{\nu-\nu_m}{2}\big(\nabla (2\psi_{j,h}^n-\psi_{j,h}^{n-1}),\nabla\chi_h\big)\nonumber\\+((2\psi_{j,h}^n-\psi_{j,h}^{n-1})\cdot\nabla v_j(t^{n+1}),\chi_h)+((2w_{j,h}^n-w_{j,h}^{n-1})\cdot\nabla\phi_{j,h}^{n+1},\chi_h)\nonumber\\-(w_{j,h}^{'n}\cdot\nabla(\phi_{j,h}^{n+1}-\phi_{j,h}^n+\phi_{j,h}^{n-1}),\chi_h)=\frac{\nu-\nu_m}{2}\big(\nabla (2\eta_{j,w}^n-\eta_{j,w}^{n-1}),\nabla\chi_h\big)\nonumber\\+\frac{\nu+\nu_m}{2}(\nabla \eta_{j,v}^{n+1},\nabla \chi_h)+((2\eta_{j,w}^n-\eta_{j,w}^{n-1})\cdot\nabla v_j(t^{n+1}),\chi_h)\nonumber\\+((2w_{j,h}^n-w_{j,h}^{n-1})\cdot\nabla\eta_{j,v}^{n+1},\chi_h)-(w_{j,h}^{'n}\cdot\nabla(\eta_{j,v}^{n+1}-\eta_{j,v}^n+\eta_{j,v}^{n-1}),\chi_h)+G_1(t,v_j,w_j,\chi_h),\label{phi2n}
\end{align}
and
\begin{align}
\bigg(\frac{3\psi_{j,h}^{n+1}-4\psi_{j,h}^n+\psi_{j,h}^{n-1}}{2\Delta t},l_h\bigg)+\frac{\nu+\nu_m}{2}(\nabla \psi_{j,h}^{n+1},\nabla l_h)+\frac{\nu-\nu_m}{2}\big(\nabla (2\phi_{j,h}^n-\phi_{j,h}^{n-1}),\nabla l_h\big)\nonumber\\+((2\phi_{j,h}^n-\phi_{j,h}^{n-1})\cdot\nabla w_j(t^{n+1}),l_h)+((2v_{j,h}^n-v_{j,h}^{n-1})\cdot\nabla\psi_{j,h}^{n+1},l_h)\nonumber\\-(v_{j,h}^{'n}\cdot\nabla(\psi_{j,h}^{n+1}-\psi_{j,h}^n+\psi_{j,h}^{n-1}),l_h)=\frac{\nu-\nu_m}{2}\big(\nabla (2\eta_{j,v}^n-\eta_{j,v}^{n-1}),\nabla l_h\big)\nonumber\\+\frac{\nu+\nu_m}{2}(\nabla \eta_{j,w}^{n+1},\nabla l_h)+((2\eta_{j,v}^n-\eta_{j,v}^{n-1})\cdot\nabla w_j(t^{n+1}),l_h)\nonumber\\+((2v_{j,h}^n-v_{j,h}^{n-1})\cdot\nabla\eta_{j,w}^{n+1},l_h)-(v_{j,h}^{'n}\cdot\nabla(\eta_{j,w}^{n+1}-\eta_{j,w}^n+\eta_{j,w}^{n-1}),l_h)+G_2(t,v_j,w_j,l_h),\label{psi2n}
\end{align}

Choose $\chi_h=\phi_{j,h}^{n+1}, l_h=\psi_{j,h}^{n+1}$ and use the identity \eqref{ident} in \eqref{phi2n} and \eqref{psi2n}, to obtain
\begin{align}
\frac{1}{4\Delta t}(\|\phi_{j,h}^{n+1}\|^2-\|\phi_{j,h}^{n}\|^2+\|2\phi_{j,h}^{n+1}-\phi_{j,h}^{n}\|^2-\|2\phi_{j,h}^{n}-\phi_{j,h}^{n-1}\|^2\nonumber\\+\|\phi_{j,h}^{n+1}-2\phi_{j,h}^{n}+\phi_{j,h}^{n-1}\|^2)+\frac{\nu+\nu_m}{2}\|\nabla \phi_{j,h}^{n+1}\|^2+\frac{\nu-\nu_m}{2}\big(\nabla (2\psi_{j,h}^n-\psi_{j,h}^{n-1}),\nabla\phi_{j,h}^{n+1}\big)\nonumber\\+((2\psi_{j,h}^n-\psi_{j,h}^{n-1})\cdot\nabla v_j(t^{n+1}),\phi_{j,h}^{n+1})-(w_{j,h}^{'n}\cdot\nabla(\phi_{j,h}^{n+1}-\phi_{j,h}^n+\phi_{j,h}^{n-1}),\phi_{j,h}^{n+1})\nonumber\\=\frac{\nu+\nu_m}{2}(\nabla \eta_{j,v}^{n+1},\nabla \phi_{j,h}^{n+1})+\frac{\nu-\nu_m}{2}\big(\nabla (2\eta_{j,w}^n-\eta_{j,w}^{n-1}),\nabla \phi_{j,h}^{n+1}\big)\nonumber\\+((2\eta_{j,w}^n-\eta_{j,w}^{n-1})\cdot\nabla v_j(t^{n+1}),\phi_{j,h}^{n+1})+((2w_{j,h}^n-w_{j,h}^{n-1})\cdot\nabla\eta_{j,v}^{n+1},\phi_{j,h}^{n+1})\nonumber\\-(w_{j,h}^{'n}\cdot\nabla(\eta_{j,v}^{n+1}-\eta_{j,v}^n+\eta_{j,v}^{n-1}),\phi_{j,h}^{n+1})+G_1(t,v_j,w_j,\phi_{j,h}^{n+1}),\label{baddddn1}
\end{align}
and
\begin{align}
\frac{1}{4\Delta t}(\|\psi_{j,h}^{n+1}\|^2-\|\psi_{j,h}^{n}\|^2+\|2\psi_{j,h}^{n+1}-\psi_{j,h}^{n}\|^2-\|2\psi_{j,h}^{n}-\psi_{j,h}^{n-1}\|^2\nonumber\\+\|\psi_{j,h}^{n+1}-2\psi_{j,h}^{n}+\psi_{j,h}^{n-1}\|^2)+\frac{\nu+\nu_m}{2}\|\nabla \psi_{j,h}^{n+1}\|^2+\frac{\nu-\nu_m}{2}\big(\nabla (2\phi_{j,h}^n-\phi_{j,h}^{n-1}),\nabla\psi_{j,h}^{n+1}\big)\nonumber\\+((2\phi_{j,h}^n-\phi_{j,h}^{n-1})\cdot\nabla w_j(t^{n+1}),\psi_{j,h}^{n+1})-(v_{j,h}^{'n}\cdot\nabla(\psi_{j,h}^{n+1}-\psi_{j,h}^n+\psi_{j,h}^{n-1}),\psi_{j,h}^{n+1})\nonumber\\=\frac{\nu+\nu_m}{2}(\nabla \eta_{j,w}^{n+1},\nabla \psi_{j,h}^{n+1})+\frac{\nu-\nu_m}{2}\big(\nabla (2\eta_{j,v}^n-\eta_{j,v}^{n-1}),\nabla \psi_{j,h}^{n+1}\big)\nonumber\\+((2\eta_{j,v}^n-\eta_{j,v}^{n-1})\cdot\nabla w_j(t^{n+1}),\psi_{j,h}^{n+1})+((2v_{j,h}^n-v_{j,h}^{n-1})\cdot\nabla\eta_{j,w}^{n+1},\psi_{j,h}^{n+1})\nonumber\\-(v_{j,h}^{'n}\cdot\nabla(\eta_{j,w}^{n+1}-\eta_{j,w}^n+\eta_{j,w}^{n-1}),\psi_{j,h}^{n+1})+G_2(t,v_j,w_j,\psi_{j,h}^{n+1}),\label{baddddn2}
\end{align}
We add equations \eqref{baddddn1} and \eqref{baddddn2}, add and subtract the term $(\nu-\nu_m)(\nabla\phi_{j,h}^{n+1},\nabla\psi_{j,h}^{n+1})$ and applying Cauchy-Schwarz results in,
\begin{align}
\frac{1}{4\Delta t}(\|\phi_{j,h}^{n+1}\|^2-\|\phi_{j,h}^{n}\|^2+\|2\phi_{j,h}^{n+1}-\phi_{j,h}^{n}\|^2-\|2\phi_{j,h}^{n}-\phi_{j,h}^{n-1}\|^2\nonumber\\+\|\psi_{j,h}^{n+1}\|^2-\|\psi_{j,h}^{n}\|^2+\|2\psi_{j,h}^{n+1}-\psi_{j,h}^{n}\|^2-\|2\psi_{j,h}^{n}-\psi_{j,h}^{n-1}\|^2\nonumber\\+\|\phi_{j,h}^{n+1}-2\phi_{j,h}^{n}+\phi_{j,h}^{n-1}\|^2+\|\psi_{j,h}^{n+1}-2\psi_{j,h}^{n}+\psi_{j,h}^{n-1}\|^2)\nonumber\\+\frac{\nu+\nu_m}{2}(\|\nabla \phi_{j,h}^{n+1}\|^2+\|\nabla \psi_{j,h}^{n+1}\|^2)\le |\nu-\nu_m|\|\nabla\phi_{j,h}^{n+1}\|\|\nabla\psi_{j,h}^{n+1}\|\nonumber\\+\frac{|\nu-\nu_m|}{2}\|\nabla( \phi_{j,h}^{n+1}-2\phi_{j,h}^n+\phi_{j,h}^{n-1})\|\|\nabla\psi_{j,h}^{n+1}\|\nonumber\\+\frac{|\nu-\nu_m|}{2}\|\nabla( \psi_{j,h}^{n+1}-2\psi_{j,h}^n+\psi_{j,h}^{n-1})\|\|\nabla\phi_{j,h}^{n+1}\|\nonumber\\+\frac{\nu+\nu_m}{2}\big(\|\nabla\eta_{j,v}^{n+1}\|\|\nabla\phi_{j,h}^{n+1}\|+\|\nabla\eta_{j,w}^{n+1}\|\|\nabla\psi_{j,h}^{n+1}\|\big)\nonumber\\+\frac{|\nu-\nu_m|}{2}\|\nabla(2\eta_{j,w}^n-\eta_{j,w}^{n-1})\|\|\nabla\phi_{j,h}^{n+1}\|+\frac{|\nu-\nu_m|}{2}\|\nabla(2\eta_{j,v}^n-\eta_{j,v}^{n-1})\|\|\nabla\psi_{j,h}^{n+1}\|\nonumber\\+|(w_{j,h}^{'n}\cdot\nabla(\phi_{j,h}^{n+1}-\phi_{j,h}^n+\phi_{j,h}^{n-1}),\phi_{j,h}^{n+1})|+|(v_{j,h}^{'n}\cdot\nabla(\psi_{j,h}^{n+1}-\psi_{j,h}^n+\psi_{j,h}^{n-1}),\psi_{j,h}^{n+1})|\nonumber\\+|((2\eta_{j,w}^n-\eta_{j,w}^{n-1})\cdot\nabla v_j(t^{n+1}),\phi_{j,h}^{n+1})|+|((2\eta_{j,v}^n-\eta_{j,v}^{n-1})\cdot\nabla w_j(t^{n+1}),\psi_{j,h}^{n+1})|\nonumber\\+|((2w_{j,h}^n-w_{j,h}^{n-1})\cdot\nabla\eta_{j,v}^{n+1},\phi_{j,h}^{n+1})|+|((2v_{j,h}^n-v_{j,h}^{n-1})\cdot\nabla\eta_{j,w}^{n+1},\psi_{j,h}^{n+1})|\nonumber\\+|(w_{j,h}^{'n}\cdot\nabla(\eta_{j,v}^{n+1}-\eta_{j,v}^n+\eta_{j,v}^{n-1}),\phi_{j,h}^{n+1})|+|(v_{j,h}^{'n}\cdot\nabla(\eta_{j,w}^{n+1}-\eta_{j,w}^n+\eta_{j,w}^{n-1}),\psi_{j,h}^{n+1})|\nonumber\\+|((2\phi_{j,h}^n-\phi_{j,h}^{n-1})\cdot\nabla w_j(t^{n+1}),\psi_{j,h}^{n+1})|+|((2\psi_{j,h}^n-\psi_{j,h}^{n-1})\cdot\nabla v_j(t^{n+1}),\phi_{j,h}^{n+1})|\nonumber\\+|G_1(t,v_j,w_j,\phi_{j,h}^{n+1})|+|G_2(t,v_j,w_j,\psi_{j,h}^{n+1})|. \label{nbound1}
\end{align}
Let us define $\alpha: = \nu+\nu_m-|\nu-\nu_m|>0$. We turn our attention to finding the bounds for the right-hand side terms in \eqref{nbound1}. Applying Cauchy-Schwarz and Young's inequalities on the first seven terms on left results in
\begin{align*}
|\nu-\nu_m|\|\nabla\phi_{j,h}^{n+1}\|\|\nabla\psi_{j,h}^{n+1}\|&\le\frac{|\nu-\nu_m|}{2}\|\nabla\phi_{j,h}^{n+1}\|^2+\frac{|\nu-\nu_m|}{2}\|\nabla\psi_{j,h}^{n+1}\|^2\\
\frac{\nu+\nu_m}{2}\|\nabla\eta_{j,v}^{n+1}\|\|\nabla\phi_{j,h}^{n+1}\|&\le\frac{\alpha}{36}\|\nabla\phi_{j,h}^{n+1}\|^2+\frac{9(\nu+\nu_m)^2}{4\alpha}\|\nabla\eta_{j,v}^{n+1}\|^2\\
\frac{\nu+\nu_m}{2}\|\nabla\eta_{j,w}^{n+1}\|\|\nabla\psi_{j,h}^{n+1}\|&\le\frac{\alpha}{36}\|\nabla\psi_{j,h}^{n+1}\|^2+\frac{9(\nu+\nu_m)^2}{4\alpha}\|\nabla\eta_{j,w}^{n+1}\|^2. 
\end{align*}
\begin{align*}
\frac{|\nu-\nu_m|}{2}&\|\nabla(\phi_{j,h}^{n+1}-2\phi_{j,h}^n+\phi_{j,h}^{n-1})\|\|\nabla\psi_{j,h}^{n+1}\|\\&\le\frac{\alpha}{36}\|\nabla\psi_{j,h}^{n+1}\|^2+\frac{9(\nu-\nu_m)^2}{4\alpha}\|\nabla(\phi_{j,h}^{n+1}-2\phi_{j,h}^n+\phi_{j,h}^{n-1})\|^2
\end{align*}
\begin{align*}
\frac{|\nu-\nu_m|}{2}&\|\nabla(\psi_{j,h}^{n+1}-2\psi_{j,h}^n+\psi_{j,h}^{n-1})\|\|\nabla\phi_{j,h}^{n+1}\|\\&\le\frac{\alpha}{36}\|\nabla\phi_{j,h}^{n+1}\|^2+\frac{9(\nu-\nu_m)^2}{4\alpha}\|\nabla(\psi_{j,h}^{n+1}-2\psi_{j,h}^n+\psi_{j,h}^{n-1})\|^2
\end{align*}
\begin{align*}
\frac{|\nu-\nu_m|}{2}\|\nabla(2\eta_{j,w}^n-\eta_{j,w}^{n-1})\|\|\nabla\phi_{j,h}^{n+1}\|&\le\frac{\alpha}{18}\|\nabla\phi_{j,h}^{n+1}\|^2+\frac{9(\nu-\nu_m)^2}{\alpha}\|\nabla\eta_{j,w}^n\|^2+\frac{9(\nu-\nu_m)^2}{4\alpha}\|\nabla\eta_{j,w}^{n-1}\|^2\\
\frac{|\nu-\nu_m|}{2}\|\nabla(2\eta_{j,v}^n-\eta_{j,v}^{n-1})\|\|\nabla\psi_{j,h}^{n+1}\|&\le\frac{\alpha}{18}\|\nabla\psi_{j,h}^{n+1}\|^2+\frac{9(\nu-\nu_m)^2}{\alpha}\|\nabla\eta_{j,v}^n\|^2+\frac{9(\nu-\nu_m)^2}{4\alpha}\|\nabla\eta_{j,v}^{n-1}\|^2
\end{align*}
Apply H\"older and Young's inequalities with \eqref{nonlinearbound} on the following eight nonlinear terms yields
\begin{align*}
|(w_{j,h}^{'n}\cdot\nabla(\phi_{j,h}^{n+1}-2\phi_{j,h}^n+\phi_{j,h}^{n-1}),\phi_{j,h}^{n+1})|&\le \frac{\alpha}{36}\|\nabla \phi_{j,h}^{n+1}\|^2+\frac{9C^2}{4\alpha}\|\nabla w_{j,h}^{'n}\|^2\|\nabla(\phi_{j,h}^{n+1}-2\phi_{j,h}^n+\phi_{j,h}^{n-1})\|^2\\
|(v_{j,h}^{'n}\cdot\nabla(\psi_{j,h}^{n+1}-2\psi_{j,h}^n+\psi_{j,h}^{n-1}),\psi_{j,h}^{n+1})|&\le \frac{\alpha}{36}\|\nabla \psi_{j,h}^{n+1}\|^2+\frac{9C^2}{4\alpha}\|\nabla v_{j,h}^{'n}\|^2\|\nabla(\psi_{j,h}^{n+1}-2\psi_{j,h}^n+\psi_{j,h}^{n-1})\|^2\\
|((2\eta_{j,w}^n-\eta_{j,w}^{n-1})\cdot\nabla v_j(t^{n+1}),\phi_{j,h}^{n+1})|&\le \frac{\alpha}{36}\|\nabla \phi_{j,h}^{n+1}\|^2+\frac{9C^2}{4\alpha}\|\nabla(2\eta_{j,w}^n-\eta_{j,w}^{n-1})\|^2\|\nabla v_j(t^{n+1})\|^2\\
|((2\eta_{j,v}^n-\eta_{j,v}^{n-1})\cdot\nabla w_j(t^{n+1}),\psi_{j,h}^{n+1})|&\le \frac{\alpha}{36}\|\nabla \psi_{j,h}^{n+1}\|^2+\frac{9C^2}{4\alpha}\|\nabla(2\eta_{j,v}^n-\eta_{j,v}^{n-1})\|^2\|\nabla w_j(t^{n+1})\|^2\\
|((2w_{j,h}^n-w_{j,h}^{n-1})\cdot\nabla\eta_{j,v}^{n+1},\phi_{j,h}^{n+1})|&\le \frac{\alpha}{36}\|\nabla \phi_{j,h}^{n+1}\|^2+\frac{9C^2}{4\alpha}\|\nabla(2w_{j,h}^n-w_{j,h}^{n-1})\|^2\|\nabla\eta_{j,v}^{n+1}\|^2\\
|((2v_{j,h}^n-v_{j,h}^{n-1})\cdot\nabla\eta_{j,w}^{n+1},\psi_{j,h}^{n+1})|&\le \frac{\alpha}{36}\|\nabla \psi_{j,h}^{n+1}\|^2+\frac{9C^2}{4\alpha}\|\nabla(2v_{j,h}^n-v_{j,h}^{n-1})\|^2\|\nabla\eta_{j,w}^{n+1}\|^2\\
|(w_{j,h}^{'n}\cdot\nabla(\eta_{j,v}^{n+1}-\eta_{j,v}^n+\eta_{j,v}^{n-1}),\phi_{j,h}^{n+1})|&\le \frac{\alpha}{36}\|\nabla \phi_{j,h}^{n+1}\|^2+\frac{9C^2}{4\alpha}\|\nabla w_{j,h}^{'n}\|^2\|\nabla(\eta_{j,v}^{n+1}-\eta_{j,v}^n+\eta_{j,v}^{n-1})\|^2\\
|(v_{j,h}^{'n}\cdot\nabla(\eta_{j,w}^{n+1}-\eta_{j,w}^n+\eta_{j,w}^{n-1}),\psi_{j,h}^{n+1})|&\le \frac{\alpha}{36}\|\nabla \psi_{j,h}^{n+1}\|^2+\frac{9C^2}{4\alpha}\|\nabla v_{j,h}^{'n}\|^2\|\nabla(\eta_{j,w}^{n+1}-\eta_{j,w}^n+\eta_{j,w}^{n-1})\|^2
\end{align*}
Apply  H\"{o}lder's inequality, Sobolev embedding theorems, Poincare's and Young's inequalities with \eqref{nonlinearbound} on the following two nonlinear terms to reveal
\begin{align*}
|((2\psi_{j,h}^n-\psi_{j,h}^{n-1})\cdot\nabla v_j(t^{n+1}),\phi_{j,h}^{n+1})|\le\frac{\alpha}{36}\|\nabla \phi_{j,h}^{n+1}\|^2+\frac{9C^2}{4\alpha}\|v_j(t^{n+1})\|_{H^2}^2\|2\psi_{j,h}^n-\psi_{j,h}^{n-1}\|^2\\
|((2\phi_{j,h}^n-\phi_{j,h}^{n-1})\cdot\nabla w_j(t^{n+1}),\psi_{j,h}^{n+1})|\le\frac{\alpha}{36}\|\nabla \psi_{j,h}^{n+1}\|^2+\frac{9C^2}{4\alpha}\|w_j(t^{n+1})\|_{H^2}^2\|2\phi_{j,h}^n-\phi_{j,h}^{n-1}\|^2
\end{align*}
Using Taylor's series, Cauchy-Schwarz, Poincare's and Young's inequalities the last two terms are evaluated as
\begin{align*}
|G_1(t,v_j,w_j,\phi_{j,h}^{n+1})|\le\frac{\alpha}{36}\|\nabla\phi_{j,h}^{n+1}\|^2+\frac{36C^2}{\alpha}\Delta t^4\|\nabla w_{j,tt}(s^*)\|^2\|\nabla v_j(t^{n+1})\|^2\\+\frac{36C^2}{\alpha} \Delta t^4\|\nabla w_{j,h}^{'n}\|^2\|\nabla v_{j,tt}(s^{**})\|^2+\frac{9(\nu-\nu_m)^2}{\alpha}\Delta t^4\|\nabla w_{j,tt}(s^{*})\|^2+\frac{4C^2}{\alpha}\Delta t^4\|v_{j,ttt}(s^{***})\|^2
\end{align*}
\begin{align*}
|G_2(t,v_j,w_j,\psi_{j,h}^{n+1})|\le\frac{\alpha}{36}\|\nabla\psi_{j,h}^{n+1}\|^2+\frac{36C^2}{\alpha}\Delta t^4\|\nabla v_{j,tt}(t^*)\|^2\|\nabla w_j(t^{n+1})\|^2\\+\frac{36C^2}{\alpha} \Delta t^4\|\nabla v_{j,h}^{'n}\|^2\|\nabla w_{j,tt}(t^{**})\|^2+\frac{9(\nu-\nu_m)^2}{\alpha}\Delta t^4\|\nabla v_{j,tt}(t^{*})\|^2+\frac{4C^2}{\alpha}\Delta t^4\|w_{j,ttt}(t^{***})\|^2
\end{align*}
with $s^{*}, s^{**}, s^{***}, t^{*}, t^{**}, t^{***}\in[t^{n-1},t^{n+1}]$. Using these estimates in \eqref{nbound1} and reducing produces
\begin{align}
\frac{1}{4\Delta t}(\|\phi_{j,h}^{n+1}\|^2-\|\phi_{j,h}^{n}\|^2+\|2\phi_{j,h}^{n+1}-\phi_{j,h}^{n}\|^2-\|2\phi_{j,h}^{n}-\phi_{j,h}^{n-1}\|^2\nonumber\\+\|\psi_{j,h}^{n+1}\|^2-\|\psi_{j,h}^{n}\|^2+\|2\psi_{j,h}^{n+1}-\psi_{j,h}^{n}\|^2-\|2\psi_{j,h}^{n}-\psi_{j,h}^{n-1}\|^2)\nonumber\\+\big[\frac{1}{4\Delta t}-\frac{9C_i(\nu-\nu_m)^2+9C_iC^2\|\nabla w_{j,h}^{'n}\|^2}{4\alpha h^2}\big]\|\phi_{j,h}^{n+1}-2\phi_{j,h}^{n}+\phi_{j,h}^{n-1}\|^2\nonumber\\+\big[\frac{1}{4\Delta t}-\frac{9C_i(\nu-\nu_m)^2+9C_iC^2\|\nabla v_{j,h}^{'n}\|^2}{4\alpha h^2}\big]\|\psi_{j,h}^{n+1}-2\psi_{j,h}^{n}+\psi_{j,h}^{n-1}\|^2\nonumber\\+\frac{\alpha}{4}(\|\nabla \phi_{j,h}^{n+1}\|^2+\|\nabla \psi_{j,h}^{n+1}\|^2)\le\frac{9(\nu+\nu_m)^2}{4\alpha}(\|\nabla\eta_{j,v}^{n+1}\|^2+\|\nabla\eta_{j,w}^{n+1}\|^2)\nonumber\\+\frac{9(\nu-\nu_m)^2}{\alpha}(\|\nabla\eta_{j,v}^{n}\|^2+\|\nabla\eta_{j,w}^{n}\|^2)+\frac{9(\nu-\nu_m)^2}{4\alpha}(\|\nabla\eta_{j,v}^{n-1}\|^2+\|\nabla\eta_{j,w}^{n-1}\|^2)\nonumber\\+\frac{9C^2}{4\alpha}\|\nabla(2\eta_{j,w}^n-\eta_{j,w}^{n-1})\|^2\|\nabla v_j(t^{n+1})\|^2+\frac{9C^2}{4\alpha}\|\nabla(2\eta_{j,v}^n-\eta_{j,v}^{n-1})\|^2\|\nabla w_j(t^{n+1})\|^2\nonumber\\+\frac{9C^2}{4\alpha}\|\nabla(2w_{j,h}^n-w_{j,h}^{n-1})\|^2\|\nabla\eta_{j,v}^{n+1}\|^2+\frac{9C^2}{4\alpha}\|\nabla(2v_{j,h}^n-v_{j,h}^{n-1})\|^2\|\nabla\eta_{j,w}^{n+1}\|^2\nonumber\\+\frac{9C^2}{4\alpha}\|\nabla w_{j,h}^{'n}\|^2\|\nabla(\eta_{j,v}^{n+1}-\eta_{j,v}^n+\eta_{j,v}^{n-1})\|^2+\frac{9C^2}{4\alpha}\|\nabla v_{j,h}^{'n}\|^2\|\nabla(\eta_{j,w}^{n+1}-\eta_{j,w}^n+\eta_{j,w}^{n-1})\|^2\nonumber\\+\frac{9C^2}{4\alpha}\|v_j(t^{n+1})\|_{H^2}^2\|2\psi_{j,h}^n-\psi_{j,h}^{n-1}\|^2+\frac{9C^2}{4\alpha}\|w_j(t^{n+1})\|_{H^2}^2\|2\phi_{j,h}^n-\phi_{j,h}^{n-1}\|^2\nonumber\\+\frac{36C^2}{\alpha}\Delta t^4\|\nabla w_{j,tt}(s^*)\|^2\|\nabla v_j(t^{n+1})\|^2+\frac{36C^2}{\alpha} \Delta t^4\|\nabla w_{j,h}^{'n}\|^2\|\nabla v_{j,tt}(s^{**})\|^2\nonumber\\+\frac{9(\nu-\nu_m)^2}{\alpha}\Delta t^4\|\nabla w_{j,tt}(s^{*})\|^2+\frac{4C^2}{\alpha}\Delta t^4\|v_{j,ttt}(s^{***})\|^2\nonumber\\+\frac{36C^2}{\alpha}\Delta t^4\|\nabla v_{j,tt}(t^*)\|^2\|\nabla w_j(t^{n+1})\|^2+\frac{36C^2}{\alpha} \Delta t^4\|\nabla v_{j,h}^{'n}\|^2\|\nabla w_{j,tt}(t^{**})\|^2\nonumber\\+\frac{9(\nu-\nu_m)^2}{\alpha}\Delta t^4\|\nabla v_{j,tt}(t^{*})\|^2+\frac{4C^2}{\alpha}\Delta t^4\|w_{j,ttt}(t^{***})\|^2
\end{align}
To make the third and fourth terms non-negative, we choose $\Delta t\le \frac{\alpha h^2}{9C_i(\nu-\nu_m)^2+9C_iC^2\max\big\{\|\nabla v_{j,h}^{'n}\|,\|\nabla w_{j,h}^{'n}\|\big\}}$. Drop the non-negative terms on the left-hand side, multiply both sides by $4\Delta t$, use the regularity assumption, $\|\phi_{j,h}^0\|=\|\psi_{j,h}^0\|=\|\phi_{j,h}^1\|=\|\psi_{j,h}^1\|=0$, $\Delta t M=T$, and sum over the time steps to find
\begin{align}
\|\phi_{j,h}^M&\|^2+\|2\phi_{j,h}^M-\phi_{j,h}^{M-1}\|^2+\|\psi_{j,h}^M\|^2+\|2\psi_{j,h}^M-\psi_{j,h}^{M-1}\|^2\nonumber\\&+\alpha\Delta t\sum\limits_{n=2}^{M}(\|\nabla\phi_{j,h}^n\|^2+\|\nabla\psi_{j,h}^n\|^2)\le C\frac{9(\nu+\nu_m)^2}{\alpha}\Delta  t\sum\limits_{n=2}^{M}(\|\nabla\eta_{j,v}^{n}\|^2+\|\nabla\eta_{j,w}^{n}\|^2)\nonumber\\&+\frac{36(\nu-\nu_m)^2}{\alpha}\Delta t\sum\limits_{n=1}^{M-1}(\|\nabla\eta_{j,v}^{n}\|^2+\|\nabla\eta_{j,w}^{n}\|^2)+\frac{9(\nu-\nu_m)^2}{\alpha}\Delta t\sum\limits_{n=1}^{M-1}(\|\nabla\eta_{j,v}^{n-1}\|^2+\|\nabla\eta_{j,w}^{n-1}\|^2)\nonumber\\&+\frac{9C^2}{\alpha}\Delta t\sum\limits_{n=1}^{M-1}\|\nabla(2\eta_{j,w}^n-\eta_{j,w}^{n-1})\|^2\|\nabla v_j(t^{n+1})\|^2+\frac{9C^2}{\alpha}\Delta t\sum\limits_{n=1}^{M-1}\|\nabla(2\eta_{j,v}^n-\eta_{j,v}^{n-1})\|^2\|\nabla w_j(t^{n+1})\|^2\nonumber\\&+\frac{9C^2}{\alpha}\Delta t\sum\limits_{n=1}^{M-1}\|\nabla(2w_{j,h}^n-w_{j,h}^{n-1})\|^2\|\nabla\eta_{j,v}^{n+1}\|^2+\frac{9C^2}{\alpha}\Delta t\sum\limits_{n=1}^{M-1}\|\nabla(2v_{j,h}^n-v_{j,h}^{n-1})\|^2\|\nabla\eta_{j,w}^{n+1}\|^2\nonumber\\&+\frac{9C^2}{\alpha}\Delta t\sum\limits_{n=1}^{M-1}\|\nabla w_{j,h}^{'n}\|^2\|\nabla(\eta_{j,v}^{n+1}-\eta_{j,v}^n+\eta_{j,v}^{n-1})\|^2\nonumber\\&+\frac{9C^2}{\alpha}\Delta t\sum\limits_{n=1}^{M-1}\|\nabla v_{j,h}^{'n}\|^2\|\nabla(\eta_{j,w}^{n+1}-\eta_{j,w}^n+\eta_{j,w}^{n-1})\|^2\nonumber\\&+\frac{9C^2}{\alpha}\Delta t\sum\limits_{n=1}^{M-1}\|v_j(t^{n+1})\|_{H^2}^2\|2\psi_{j,h}^n-\psi_{j,h}^{n-1}\|^2+\frac{9C^2}{\alpha}\Delta t\sum\limits_{n=1}^{M-1}\|w_j(t^{n+1})\|_{H^2}^2\|2\phi_{j,h}^n-\phi_{j,h}^{n-1}\|^2\nonumber\\&+\frac{144C^2}{\alpha} \Delta t^4\Delta t\sum\limits_{n=1}^{M-1}(\|\nabla w_{j,h}^{'n}\|^2\|\nabla v_{j,tt}(t^{**})\|^2+\|\nabla v_{j,h}^{'n}\|^2\|\nabla w_{j,tt}(t^{**})\|^2)\nonumber\\&+\frac{144C^2}{\alpha}\Delta t^4\Delta t\sum\limits_{n=1}^{M-1}(\|\nabla w_{j,tt}(s^*)\|^2\|\nabla v_j(t^{n+1})\|^2+\|\nabla v_{j,tt}(t^*)\|^2\|\nabla w_j(t^{n+1})\|^2)\nonumber\\&+\frac{9(\nu-\nu_m)^2}{\alpha}\Delta t^4\Delta t\sum\limits_{n=1}^{M-1}(\|\nabla w_{j,tt}(s^{*})\|^2+\|\nabla v_{j,tt}(t^{*})\|^2)\nonumber\\&+\frac{16C^2}{\alpha}\Delta t^4\Delta t\sum\limits_{n=1}^{M-1}(\|w_{j,ttt}(t^{***})\|^2+\|v_{j,ttt}(s^{***})\|^2)
\end{align}

Applying the regularity assumptions, stability bound, interpolation estimates for $v_j$, $w_j$
\begin{align}
&\|\phi_{j,h}^M\|^2+\|2\phi_{j,h}^M-\phi_{j,h}^{M-1}\|^2+\|\psi_{j,h}^M\|^2+\|2\psi_{j,h}^M-\psi_{j,h}^{M-1}\|^2+\alpha\Delta t\sum\limits_{n=2}^{M}(\|\nabla\phi_{j,h}^n\|^2+\|\nabla\psi_{j,h}^n\|^2)\nonumber\\&\le C\frac{9(\nu+\nu_m)^2}{\alpha}h^{2k}+C\frac{45(\nu-\nu_m)^2}{\alpha}h^{2k}+\frac{9C}{\alpha}\Delta t\sum\limits_{n=1}^{M-1}(\|2\phi_{j,h}^n-\phi_{j,h}^{n-1}\|^2+\|2\psi_{j,h}^n-\psi_{j,h}^{n-1}\|^2)\nonumber\\&+\frac{9C}{\alpha}h^{2k}+\frac{304C}{\alpha}\Delta t^4 +\frac{9C(\nu-\nu_m)^2}{\alpha}\Delta t^4
\end{align}
Applying the discrete Gronwall lemma, we have
\begin{align}
\|\phi_{j,h}^M\|^2+&\|2\phi_{j,h}^M-\phi_{j,h}^{M-1}\|^2+\|\psi_{j,h}^M\|^2+\|2\psi_{j,h}^M-\psi_{j,h}^{M-1}\|^2+\alpha\Delta t\sum\limits_{n=2}^{M}(\|\nabla\phi_{j,h}^n\|^2+\|\nabla\psi_{j,h}^n\|^2)\nonumber\\&\le \frac{C}{\alpha} e^{\frac{9TC}{\alpha}}((\nu^2+\nu_m^2+1)h^{2k}+((\nu-\nu_m)^2+1)\Delta t^4)
\end{align}
Using the triangular inequality allows us to write
\begin{align}
&\|e_{j,v}^M\|^2+\|e_{j,w}^M\|^2+\alpha\Delta t\sum\limits_{n=2}^{M}(\|\nabla e_{j,v}^n\|^2+\|\nabla e_{j,w}^n\|^2)\le 2\big(\|\phi_{j,h}^M\|^2+\|\psi_{j,h}^M\|^2\nonumber\\&+\alpha\Delta t\sum\limits_{n=2}^{M}(\|\nabla\phi_{j,h}^n\|^2+\|\nabla\psi_{j,h}^n\|^2)+\|\eta_{j,v}^M\|^2+\|\eta_{j,w}^M\|^2+\alpha\Delta t\sum\limits_{n=2}^{M}(\|\nabla\eta_{j,v}^n\|^2+\|\nabla\eta_{j,w}^n\|^2)\big)\nonumber\\&\le\frac{C}{\alpha} e^{\frac{9TC}{\alpha}}((\nu^2+\nu_m^2+1)h^{2k}+((\nu-\nu_m)^2+1)\Delta t^4)+Ch^{2k+2}
\end{align}
Now summing over $j$ and using the triangular inequality completes the proof.
\end{proof}
\section{\large Numerical Experiments:}
To test the proposed algorithm \eqref{Algn1} and theory, in this section we present results of numerical experiments. In all experiments, we used $((Q_2)^2,Q_1, (Q_2)^2, Q_1)$ Taylor Hood finite elements on regular quadrilateral meshes and open source finite element library DealII\cite{dealII85}.
\subsection{\small Convergence Rate Verification:}
To verify the predicted convergence rates of our analysis in section \ref{ErrorAnalysis}, we begin this experiment with a manufactured analytical solution,
\[
{v}=\left(\begin{array}{c} \cos y+(1+t)\sin y \\ \sin x+(1+t)\cos x \end{array} \right), \
{w}=\left(\begin{array}{c} \cos y-(1+t)\sin y \\ \sin x-(1+t)\cos x \end{array} \right), \ p =(x-y)(1+t), \ \lambda=0,
\]
on the domain $\Omega = (0,1)^2$.
Next, to create four different true solutions, we perturb the above solution introducing a parameter $\epsilon$ and defining as follows: $v_j:=\begin{cases} 
      (1+(-1)^{j-1}\epsilon)v & 1\le j<3 \\
      (1+(-1)^{j-1}2\epsilon)v & 3\le j\le 4
   \end{cases}$, similarly for $w_j$, where $j\in\mathbb{N}$. Using these perturbed solutions, we compute right-hand side forcing terms. We consider the initial conditions $v_j(0)$ and $w_j(0)$. On the boundary of the unit square, Dirichlet conditions are used. The algorithm \ref{Algn1} computes the discrete ensemble average $<v_h^n>$ and $<w_h^n>$, and these will be used to compare to the true average $<v(t^n)>$ and $<w(t^n)>$ respectively. We notate the ensemble average error as $<e_u>:=<u_h>^n-<u(t^n)>$. For our choice of elements, the theory predicts the $L^2(0, T;H^1(\Omega)^d)$ error to be $O(h^2+\Delta t^2)$ provided $\Delta t<O(h^2)$. We consider three different choices $\epsilon=10^{-3},10^{-2}\text{ } \text{and}\text{ }10^{-1}$ for the perturbation parameter herein and end time $T=0.001$ for this test. For these choice of $\epsilon$, Tables \ref{convergence1}-\ref{convergence2} exhibit errors and convergence rates, and we observe second order convergence of our scheme.
\begin{table}[h!]
  \begin{center}
  \scriptsize\begin{tabular}{|c|c|c|c|c|c|c|c|}\hline
    \multicolumn{2}{|c|}{}&\multicolumn{2}{|c|}{$\epsilon=0.001$}
    & \multicolumn{2}{|c|}{$\epsilon=0.01$}& \multicolumn{2}{|c|}{$\epsilon=0.1$} \\ \hline
   $h$ &$\Delta t$ & $\|<e_v>\|_{2,1}$ & rate   &$\|<e_v>\|_{2,1}$  & rate & $\|<e_v>\|_{2,1}$ & rate  \\ \hline
   $\frac{1}{2}$  & $\frac{T}{4}$ & 3.650e-4 &      & 3.64973e-4 &      & 3.64973e-4 & \\ \hline
   $\frac{1}{4}$  & $\frac{T}{8}$ & 1.008e-4 & 1.86 & 1.00764e-4 & 1.86 & 1.00764e-4  & 1.86         \\ \hline
   $\frac{1}{8}$  & $\frac{T}{16}$& 2.621e-5 & 1.94 & 2.62134e-5 & 1.94 & 2.62134e-5  & 1.94         \\ \hline
   $\frac{1}{16}$ & $\frac{T}{32}$& 6.670e-6 & 1.97 & 6.67033e-6 & 1.97 & 6.67034e-6  & 1.97   \\ \hline
   $\frac{1}{32}$ & $\frac{T}{64}$& 1.683e-6 & 1.99 & 1.69718e-6 & 1.97 & 1.72669e-6  & 1.95   \\ \hline
     \end{tabular}
  \end{center}
  \caption{\label{convergence1}\footnotesize Error and  convergence rates for $v$ with $\nu=0.01$, $\nu_m=0.001$.}
  \end{table}
  
  \begin{table}[h!]
    \begin{center}
    \scriptsize\begin{tabular}{|c|c|c|c|c|c|c|c|}\hline
      \multicolumn{2}{|c|}{ } &\multicolumn{2}{|c|}{$\epsilon=0.001$}
      & \multicolumn{2}{|c|}{$\epsilon=0.01$}& \multicolumn{2}{|c|}{$\epsilon=0.1$} \\ \hline
     $h$ &$\Delta t$ & $\|<e_w>\|_{2,1}$  & rate   &$\|<e_w>\|_{2,1}$  & rate & $\|<e_w>\|_{2,1}$ & rate  \\ \hline
     $\frac{1}{2}$  & $\frac{T}{4}$ & 7.168e-4 &      & 7.168e-4&      & 7.168e-4 & \\ \hline
     $\frac{1}{4}$  & $\frac{T}{8}$ & 1.930e-4 & 1.89 & 1.930e-4& 1.89 & 1.930e-4 & 1.89         \\ \hline
     $\frac{1}{8}$  & $\frac{T}{16}$& 4.992e-5 & 1.95 & 4.992e-5& 1.95 & 4.992e-5 & 1.95         \\ \hline
     $\frac{1}{16}$ & $\frac{T}{32}$& 1.268e-5 & 1.98 & 1.268e-5& 1.98 & 1.268e-5 & 1.98   \\ \hline
     $\frac{1}{32}$ & $\frac{T}{64}$& 3.196e-6 & 1.99 & 3.197e-6& 1.99 & 3.197e-6 & 1.99   \\ \hline
     
     \end{tabular}
    \end{center}
    \caption{\label{convergence2}\footnotesize Error and  convergence rates for $w$ with $\nu=0.01$, $\nu_m=0.001$.}
    \end{table}

\subsection{\small MHD Channel Flow over a Step:}

Next, we consider a domain which is a $40\times 10$ rectangular channel with a $1\times 1$ step five units away from the inlet into the channel. No slip boundary condition is prescribed for the velocity and $B=<0,1>^T$  is enforced for the magnetic field on the walls and step, $u=<y(10-y)/25, 0>^T$ and $B=<0,1>^T$ at the inlet and outlet. 


An ensemble of four different solutions with the corresponding perturbed initial conditions $u_j(0)$ and $B_j(0)$ and perturbed inflow and outflow are considered. As we used second order BDF-2 scheme to approximate time derivative, we used backward-Euler method at the first time step to get the second initial condition. A mesh of the domain with $44k$ velocity degrees of freedom is shown in figure \ref{mesh}. The simulations of the algorithm \ref{Algn1} are done with the various values of $\epsilon$.

\begin{figure}[h!]
\begin{center}\vspace{-3in}
			\includegraphics[width = 1.25\textwidth, height=.65\textwidth,viewport=-22 0 500 260, clip]{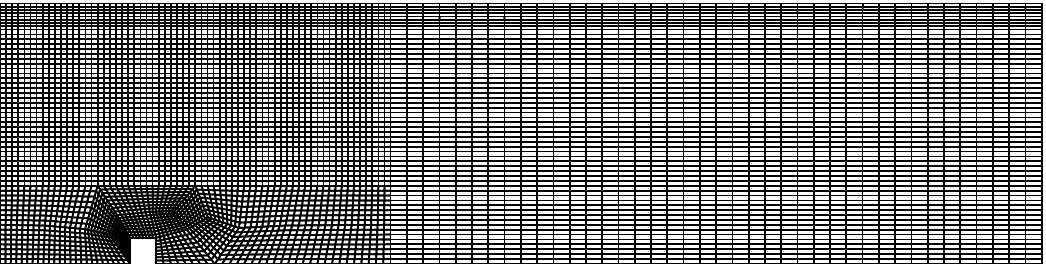}
	 	 	\caption{
			Mesh for the channel flow with a step example.
			}\label{mesh}
\end{center}
\end{figure}

\begin{figure}[h!]
\begin{center}
			\includegraphics[width = 1.\textwidth, height=.25\textwidth,clip]{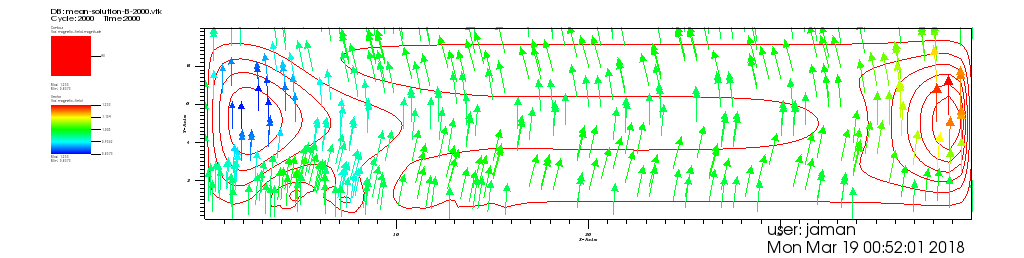}
	 	 	\caption{Shown above is $T=2$, ensemble magnetic field solution and magnetic field contour for MHD channel flow over a step with $\Delta t=0.001$, $\epsilon=0.001$, $\nu=0.001$ and $\nu_m=1$.
			}
\end{center}
\end{figure}

\begin{figure}[h!]
\begin{center}
			\includegraphics[width = 1.\textwidth, height=.25\textwidth,clip]{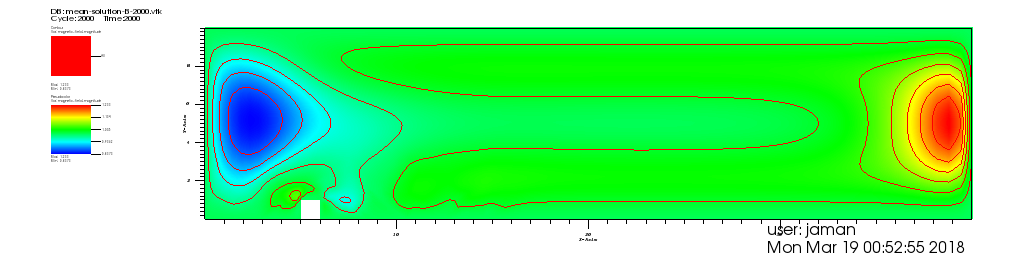}
	 	 	\caption{Shown above is $T=2$, magnitudes of ensemble magnetic field solutions (magnetic) for MHD channel flow over a step with $\Delta t=0.001$, $\epsilon=0.001$, $\nu=0.001$, and $\nu_m=1$.
			}
\end{center}
\end{figure}

\begin{figure}[h!]
\begin{center}
			\includegraphics[width = 1.\textwidth, height=.25\textwidth,clip]{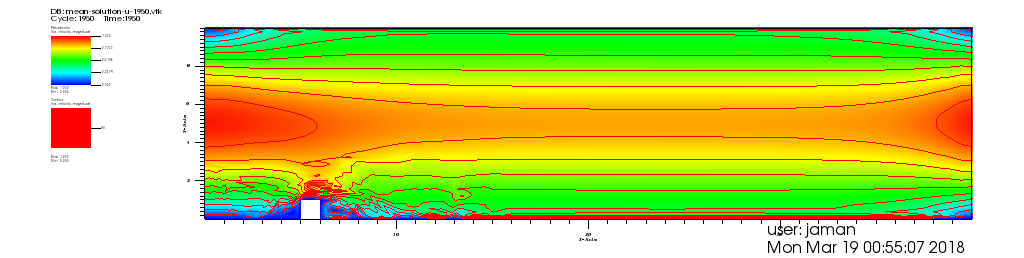}
	 	 	\caption{Shown above is $T=2$, velocity ensemble solutions (shown as streamlines over speed contours) for MHD channel flow over a step with $\Delta t=0.001$, $\epsilon=0.001$, $\nu=0.001$ and $\nu_m=1$.
			}
\end{center}
\end{figure}

\section{Conclusion:}
This paper represents an efficient second order method for computing MHD flow ensemble with noisy input data. The algorithm combines the breakthrough idea of Trenchea \cite{T14} to present a decoupled stable scheme in terms of Els\"asser variables and  the breakthrough idea for efficient computation of flow ensemble for Navier-Stokes \cite{JL14} and extends it to MHD. This work is also an extension of the author's first order accurate work \cite{MR17} for computing MHD flow ensemble. The key features to the efficiency of the algorithm are
(i) it is second order accurate stable decoupled method-split into two Oseen problems, which are much easier to solve and can be solved simultaneously (ii) at each time step, all $J$ different linear systems share the same coefficient matrix, as a result storage requirement is reduced, a single assembly of the coefficient matrix is required instead of $J$ times, preconditioners need to build once and can be reused.

We proved the stability and second order convergence of the algorithm with respect to the time size, which is an improvement from the author's earlier work of a first order scheme for computing MHD flow ensemble. The couple MHD system is split into two Oseen sub-problems at each time step where in the schemes the nonlinearities are treated explicitly at each time step.  Numerical experiments were done on a unit square with a manufactured solution that verified the predicted convergence rates. Finally, we applied our scheme on a benchmark channel flow over a step problem and showed the method performed well.

Reduced order modeling (ROM) for the ensemble MHD flow computation will be the future work. Recently, it has been shown the data-driven filtered ROM for flow problem \cite{XMRI17} works well for the complex system. To reduced computation cost further to simulate ensemble MHD system as well as more accurate results, it is worth exploring in ROM with physically accurate data.

\bibliographystyle{plain}
\bibliography{Second_order_ensembleMHD_Mohebujjaman}
\end{document}